\documentclass[preprint,12pt]{elsarticle}

\usepackage[T1]{fontenc}
\usepackage[utf8]{inputenc}
\usepackage{lmodern}
\usepackage{microtype}
\usepackage{amsmath,amssymb,amsthm,mathtools}
\usepackage{booktabs}
\usepackage{enumitem}
\usepackage[colorlinks=true,linkcolor=blue,citecolor=blue,urlcolor=blue]{hyperref}

\journal{Finite Fields and Their Applications}

\theoremstyle{plain}
\newtheorem{theorem}{Theorem}[section]
\newtheorem{proposition}[theorem]{Proposition}
\newtheorem{corollary}[theorem]{Corollary}
\newtheorem{lemma}[theorem]{Lemma}

\theoremstyle{definition}
\newtheorem{definition}[theorem]{Definition}

\theoremstyle{remark}
\newtheorem{remark}[theorem]{Remark}

\DeclareMathOperator{\Gal}{Gal}
\DeclareMathOperator{\Fix}{Fix}
\DeclareMathOperator{\Frob}{Frob}
\DeclareMathOperator{\AGL}{AGL}
\DeclareMathOperator{\Aut}{Aut}
\DeclareMathOperator{\rank}{rank}
\DeclareMathOperator{\Nm}{Nm}

\newcommand{\F}{\mathbb F}
\newcommand{\Pone}{\mathbb P^1}
\newcommand{\Aone}{\mathbb A^1}
\newcommand{\cO}{\mathcal O}

\numberwithin{equation}{section}

\begin{document}

\begin{frontmatter}

\title{Affine monodromy and exact value distributions over finite fields}

\author{David Kumallagov}

\begin{abstract}
We study finite field value distributions through the fixed-point statistics
of monodromy groups.  For a regular degree-\(n\) cover, omitted values are
controlled by derangements.  Thus natural symmetric monodromy gives support
density \(1-D_n/n!\to 1-e^{-1}\), while the Cameron--Cohen bound gives the
universal ceiling \(1-1/n\), attained by sharply \(2\)-transitive affine
monodromy.

We give an explicit polynomial realization of this optimal mechanism.  For
\(N=p^e\) and \(h\mid N-1\), set
\[
  \Lambda_{N,h}(U)=U\bigl(U^{(N-1)/h}-1\bigr)^h .
\]
Its geometric Galois closure is rational,
\[
  U=z^h,
  \qquad
  T=(z^N-z)^h,
\]
and its geometric monodromy is the affine group
\((\F_N,+)\rtimes H_h\).  For every extension
\(\F_Q/\F_p\) we compute the complete fibre enumerator of
\(\Lambda_{N,h}\) exactly, including the nonregular cases.
In the full affine case \(h=N-1\), the polynomial
\[
  U(U-1)^{N-1}
\]
attains the Wan--Shiue--Chen upper bound for non-permutation polynomials
over every finite field containing \(\F_N\); over arbitrary extensions we
compute the exact defect from that bound.
\end{abstract}


\end{frontmatter}

\section{Introduction}
\label{sec:intro}

Let \(f\in\F_Q[U]\).  Its value set $V_f=f(\F_Q)$
and its fibre statistics
\[
  r_Q(t)=\#\{u\in\F_Q:f(u)=t\},\quad t\in\F_Q,
\]
measure how far \(f\) is from being a permutation polynomial.  Value sets
of polynomials over finite fields have been studied since Cohen's work on
polynomial distributions \cite{Cohen1970,Cohen1981}.  The sharp upper
bound for a non-permutation polynomial of degree \(d\) used below is the
one-variable case of Mullen--Wan--Wang
\cite[Theorem~2.1]{MullenWanWang2012}, building on Wan--Shiue--Chen
\cite[Theorem~1]{WanShiueChen1993}:
\begin{equation}
  \#V_f\le Q-\left\lceil\frac{Q-1}{d}\right\rceil.
  \label{eq:intro-wsc}
\end{equation}

A useful way to understand such questions is through monodromy.  If
\(\pi:X\to Y\) is a finite separable cover over a finite field, then
Chebotarev relates the distribution of finite-field fibres to fixed
points of Frobenius elements in the monodromy action.  In particular,
a target value is omitted precisely when the corresponding Frobenius
acts without fixed points on the geometric generic fibre.  Thus omitted
values are governed by derangements.

This viewpoint shows that large degree alone does not make a
one-dimensional cover nearly surjective.  For a cover with natural
symmetric monodromy \(S_n\), the limiting support density is
\[
  1-\frac{D_n}{n!},
\]
where \(D_n\) is the number of derangements in \(S_n\); this tends to
\(1-e^{-1}\).  The opposite mechanism is special affine
monodromy. The Cameron--Cohen derangement bound
\cite[Theorem~1]{CameronCohen1992} gives
\[
  \delta(G,\Omega)\ge \frac1n
\]
for every nontrivial transitive action of degree \(n\), with equality
exactly in the sharply \(2\)-transitive case.  Consequently the largest
possible regular support density in degree \(n\) is \(1-1/n\), and it is
attained by sharply \(2\)-transitive affine actions.

The purpose of this paper is to give an explicit polynomial realization
of this optimal affine mechanism and to compute its finite-field fibres
exactly.  Let
\[
  N=p^e,\qquad h\mid N-1,\qquad A=\frac{N-1}{h},
\]
and define
\begin{equation}
  \Lambda_{N,h}(U)=U(U^A-1)^h\in\F_p[U].
  \label{eq:intro-Lambda}
\end{equation}
The basic identity
\begin{equation}
  U=z^h,
  \quad
  T=(z^N-z)^h=U(U^A-1)^h
  \label{eq:intro-galois-closure}
\end{equation}
shows that the geometric Galois closure of
\(T=\Lambda_{N,h}(U)\) is rational and has affine monodromy
\[
  G_{N,h}=(\F_N,+)\rtimes H_h,
\]
where \(H_h\le\F_N^\times\) is the subgroup of order \(h\).  For
\(h=N-1\) this is the full affine group \(\AGL(1,N)\), the standard
sharply \(2\)-transitive affine action.

Affine-type polynomial monodromy is classical in the theory of
exceptional polynomials and primitive polynomial monodromy groups; see
\cite{FriedGuralnickSaxl1993,GuralnickMuller1997,GuralnickSaxl1995,
Muller1995,Guralnick2008}.  Also, over a fixed finite field, the maps
\(\Lambda_{N,h}\) are special first-order generalized cyclotomic mappings;
see \cite{WanLidl1991}, \cite[Theorems~1.1.3 and~1.1.4]{BorsWang2021},
\cite[Definitions~2.1--2.2 and Theorem~2.4]{ZhengEtAl2025}, and
\cite[Theorems~1--3]{GaoWang2014}.  Our main contribution
is the explicit scalar-affine Galois closure above, together with a complete
 fibre law and its consequences for optimal support.

The main exact theorem is as follows.  Let \(Q=p^m\), and put
\[
  s=p^{\gcd(e,m)},\qquad
  C=\frac{N-1}{s-1},\qquad
  \kappa=\gcd(h,C).
\]
Then
\[
  r_Q(0)=1+\frac{\kappa(s-1)}h.
\]
Among nonzero target values the only possible fibre sizes are \(0\),
\(1\), and \(s\), with exact multiplicities
\[
  \frac{\kappa(s-1)Q}{hs},\qquad
  \frac{(h-\kappa)(Q-1)}h,\qquad
  \frac{\kappa(Q-s)}{hs},
\]
respectively.  In particular,
\[
  \#\Lambda_{N,h}(\F_Q)
  =
  Q\left(1-\frac{\kappa(s-1)}{hs}\right).
\]
In the full affine case \(h=N-1\), this becomes
\[
  \#\Lambda_N(\F_Q)=Q\left(1-\frac1s\right),
  \qquad
  \Lambda_N(U)=U(U-1)^{N-1}.
\]
Hence if \(\F_N\subseteq\F_Q\), then \(s=N\) and
\[
  \#\Lambda_N(\F_Q)
  =
  Q-\frac QN
  =
  Q-\left\lceil\frac{Q-1}{N}\right\rceil.
\]
Thus \(\Lambda_N\) is a non-permutation polynomial attaining the
Wan--Shiue--Chen upper bound over every finite field containing \(\F_N\).
Over arbitrary extensions we compute the exact defect from this bound.

The final section compares this affine mechanism with the generic
symmetric one.  Symmetric monodromy has support density tending to
\(1-e^{-1}\), while the full affine family realizes the optimal density
\(1-1/N\).  Thus the paper gives a concrete finite-field model showing
how special affine monodromy, rather than increasing degree with generic
symmetric monodromy, produces nearly maximal value sets.

The paper is organized as follows.  Section~\ref{sec:fixedpoint} recalls
the fixed-point Chebotarev framework, including constant field cosets,
derangements, and affine fixed-point profiles.  Section~\ref{sec:lambda-monodromy}
computes the geometric and arithmetic monodromy of \(\Lambda_{N,h}\).
Section~\ref{sec:exact} proves the exact all-extension fibre law, and
Section~\ref{sec:fibre} derives the fibre enumerator, moments, value-set
size, and Wan--Shiue--Chen defect.  Section~\ref{sec:pullback} records
coprime pullbacks to elliptic bases.  Section~\ref{sec:symmetric}
compares the affine and symmetric monodromy mechanisms.

\medskip
\noindent\textbf{Conventions.} All finite fields of characteristic $p$ are regarded as subfields of a fixed algebraic closure $\overline{\F}_p$.  Thus $\F_q\cap\F_N=\F_{p^{\gcd(m,e)}}$ for $q=p^m$ and $N=p^e$, and $\F_q\cdot\F_N$ denotes the compositum $\F_{p^{\operatorname{lcm}(m,e)}}$.  Fibres over finite fields are counted as sets of rational points, without multiplicity.

\section{Fixed points, constant fields, and affine profiles}
\label{sec:fixedpoint}

\subsection{Covers, monodromy, and Frobenius cosets}

We recall some notions used throughout the paper.

\begin{definition} \label{def:monodromy}
Let $k=\F_q$, let $X$ and $Y$ be smooth projective geometrically connected curves over $k$, and let $\pi:X\longrightarrow Y$
be a finite separable morphism of degree $n$.  Put $K=k(Y)$ and $E=k(X)$.  Let $L/K$ be the Galois closure of the finite separable extension $E/K$.  The \emph{arithmetic monodromy group} is
\[
  A_\pi=\Gal(L/K).
\]
Let $k_L=L\cap\overline{k}$ be the algebraic closure of $k$ in $L$.  The \emph{geometric monodromy group} is
\[
  G_\pi=\Gal(L/k_L(Y)).
\]
Then $G_\pi\triangleleft A_\pi$ and
\[
  A_\pi/G_\pi\simeq \Gal(k_L/k).
\]
We say that the Galois closure is \emph{regular over $k$} if $k_L=k$.
The group $A_\pi$ acts on the set 
\[
  \Omega=\operatorname{Hom}_K(E,L),
\]
in a standard way.
Since $E/K$ is separable of degree $n$, the set
$\operatorname{Hom}_K(E,\overline K)$ has cardinality $n$, and the Galois
closure $L$ contains the images of all these embeddings.  Thus
$|\Omega|=n$.  Equivalently, $\Omega$ is the set of geometric points of
the geometric generic fibre of $\pi$.
\end{definition}

For a polynomial $f\in k[U]$ of degree $n$ with $f'(U)\ne0$, this definition says that $L$ is the splitting field of $f(U)-T$ over $k(T)$ and $\Omega$ is the set of its $n$ roots over an algebraic closure.  The geometric monodromy is obtained after extending constants to $\overline{k}$, or equivalently after replacing $k$ by $k_L$ in the Galois closure.

\begin{definition}
\label{def:branch}
A point $y\in Y(\overline{k})$ is a \emph{branch value} of $\pi$ if some point above $y$ has ramification index greater than one.  If $Y=\Pone_T$, a branch value lying in $\Aone_T$ is called a \emph{finite branch value}.  For $Q=q^m$ and $y\in Y(\F_Q)$, put
\[
  r_Q(y)=\#\pi^{-1}(y)(\F_Q),
\]
where geometric points in the fibre are counted without multiplicity.  For $0\le j\le n$, put
\[
  N_j(Q)=\#\{y\in Y(\F_Q):r_Q(y)=j\}.
\]
\end{definition}

Let $k_L=\F_{q^c}$.  Let $\bar\phi\in A_\pi/G_\pi$ be the image of arithmetic Frobenius on constants.  For $m\ge1$ define the Frobenius coset
\begin{equation}
  A_m=\{\sigma\in A_\pi:\sigma G_\pi=\bar\phi^m\}.
  \label{eq:frob-coset}
\end{equation}
Thus $A_m$ is a coset of $G_\pi$ in $A_\pi$ and depends only on $m\bmod c$.  For $\sigma\in A_\pi$ write
\[
  F(\sigma)=\#\Fix_\Omega(\sigma).
\]

\begin{theorem}
\label{thm:coset-chebotarev}
With the notation above, for every $0\le j\le n$ and $Q=q^m$,
\begin{equation}
  N_j(Q)=\rho_j(m)\#Y(\F_Q)+O_\pi(Q^{1/2}),
  \label{eq:coset-cheb}
\end{equation}
where
\begin{equation}
  \rho_j(m)=\frac1{|G_\pi|}\#\{\sigma\in A_m:F(\sigma)=j\}.
  \label{eq:rhoj-coset}
\end{equation}
  If $k$ is algebraically closed in $L$, then $A_\pi=G_\pi$ and
\[
  \rho_j=\frac1{|G_\pi|}\#\{\sigma\in G_\pi:F(\sigma)=j\}.
\]
\end{theorem}

\begin{proof}
Let $B\subset Y(\overline{k})$ be the finite branch locus.  We first count
points in $Y(\F_Q)\setminus B$.  For such a point $y$, the Frobenius
conjugacy class $\Frob_y$ is defined in $A_\pi$.  Since $y$ is
$\F_Q$-rational, its action on the constant field $k_L$ is
$a\mapsto a^Q=a^{q^m}$; hence $\Frob_y$ lies in the coset $A_m$.

The standard specialization argument identifies the number of
$\F_Q$-rational points in the fibre with the number of fixed points of
Frobenius on the geometric generic fibre $r_Q(y)=\#\Fix_\Omega(\Frob_y)$.
Applying the function-field Chebotarev theorem (cf. \cite[Theorems~9.13A and~9.13B]{Stichtenoth2009} and 
\cite[Theorem~1.1 and Corollary~1.2]{Kosters2014}) to the union of conjugacy
classes in $A_m$ on which $F(\sigma)=j$  gives
\[
  \#\{y\in Y(\F_Q)\setminus B:F(\Frob_y)=j\}
  =\rho_j(m)\#Y(\F_Q)+O_\pi(Q^{1/2}).
\]
The finitely many branch values contribute $O_\pi(1)$, which is absorbed
in the error term.  This proves our claimes. 
\end{proof}

\begin{remark}
1) The ``constant field correction'' is essential.  If $k_L\ne k$, the limiting distribution is not the fixed-point distribution of the geometric group $G_\pi$ itself, but the periodic sequence of fixed-point distributions on the cosets $A_m$.  This is the same arithmetic phenomenon that occurs in exceptional covers.

2) For applications of this fixed-point and derangement viewpoint to covers of
curves over finite fields, see also
\cite[Theorem~1.1]{GuralnickWan1997}; for the general specialization
framework see \cite[Proposition~6.4.8]{FriedJarden2008}.
\end{remark}

\subsection{Distributional consequences}

For a transitive permutation group $G$ on a finite set $\Omega$, define
\[
  \rho_j(G,\Omega)=\frac1{|G|}\#\{g\in G:\#\Fix_\Omega(g)=j\}.
\]
The quantity
\[
  \delta(G,\Omega)=\rho_0(G,\Omega)
\]
is the \emph{derangement proportion}.  In the regular case of Theorem~\ref{thm:coset-chebotarev}, it is the limiting proportion of omitted target values.

\begin{corollary}
\label{cor:distributional}
Under the hypotheses of Theorem~\ref{thm:coset-chebotarev},
\begin{equation}
  \frac{\#\pi(X(\F_Q))}{\#Y(\F_Q)}
  =1-\rho_0(m)+O_\pi(Q^{-1/2}).
  \label{eq:support-general}
\end{equation}
Moreover, for each fixed integer $a\ge1$,
\begin{equation}
  \frac1{\#Y(\F_Q)}\sum_{y\in Y(\F_Q)} r_Q(y)^a
  =\frac1{|G_\pi|}\sum_{\sigma\in A_m}F(\sigma)^a+O_{\pi,a}(Q^{-1/2}).
  \label{eq:fibre-moments-general}
\end{equation}
\end{corollary}

\begin{proof}
This is obtained by summing the asymptotic formula \eqref{eq:coset-cheb} over $j\ge1$ for the support and over all $j$ with weight $j^a$ for the moments.
\end{proof}

\begin{definition}
Let $G$ act transitively on $\Omega$, $|\Omega|=n$.  The \emph{rank} of the action is
\[
  \rank(G,\Omega)=\#(G\backslash(\Omega\times\Omega)),
\]
the number of orbits on ordered pairs.  Equivalently, it is the number of orbits of a point stabilizer $G_\alpha$ on $\Omega$.  The action is \emph{sharply $2$-transitive} if for any ordered pairs $(\alpha,\beta)$ and $(\alpha',\beta')$ with $\alpha\ne\beta$ and $\alpha'\ne\beta'$, there is a unique $g\in G$ satisfying $g\alpha=\alpha'$ and $g\beta=\beta'$.
\end{definition}

\begin{theorem}
\label{thm:CC}
Let $G$ be a nontrivial transitive permutation group of degree $n$ and rank $r$.  Then
\begin{equation}
  \delta(G,\Omega)\ge \frac{r-1}{n}\ge\frac1n.
  \label{eq:CC-bound}
\end{equation}
Moreover, the case $\delta(G,\Omega)=1/n$ occurs if and only if the action is sharply $2$-transitive; see  \cite[Theorem~1]{CameronCohen1992}.
\end{theorem}

\begin{corollary}
\label{cor:ceiling}
Let $\pi:X\to Y$ be a degree $n$ cover whose Galois closure is regular over $\F_q$ in the sense of Definition~\ref{def:monodromy}.  If $G$ is the geometric monodromy group and $r=\rank(G,\Omega)$, then
\[
  \frac{\#\pi(X(\F_Q))}{\#Y(\F_Q)}
  \le 1-\frac{r-1}{n}+O_\pi(Q^{-1/2}).
\]
In particular, every degree-$n$ cover with regular Galois closure has limiting support at most $1-1/n$, and equality in the main term is possible precisely for sharply $2$-transitive monodromy.
\end{corollary}

\begin{proof}
Combine Corollary~\ref{cor:distributional} in the regular case with Theorem~\ref{thm:CC}.
\end{proof}

\subsection{Affine fixed-point profiles}

Let $V$ be a finite vector space,  and let $H\le\operatorname{GL}(V)$.  The affine group
\[
  G=V\rtimes H
\]
acts on $V$ by
\[
  (b,A):x\longmapsto Ax+b.
\]
For $A\in H$ put
\[
  k(A)=|\ker(A-I)|.
\]

\begin{proposition}
\label{prop:affine-enumerator}
The fixed-point enumerator in the affine action of $V\rtimes H$ is
\begin{equation}
  \Phi_G(Z)=\frac1{|H|}\sum_{A\in H}
  \left(1-\frac1{k(A)}+\frac{Z^{k(A)}}{k(A)}\right).
  \label{eq:affine-enumerator}
\end{equation}
Thus
\begin{equation}
  \delta(G,V)=1-\frac1{|H|}\sum_{A\in H}\frac1{k(A)},
  \label{eq:affine-derangement}
\end{equation}
and for every integer $a\ge1$,
\begin{equation}
  \frac1{|G|}\sum_{g\in G}\#\Fix_V(g)^a
  =\frac1{|H|}\sum_{A\in H}k(A)^{a-1}.
  \label{eq:affine-moments}
\end{equation}
\end{proposition}

\begin{proof}
For fixed linear part $A$, the fixed-point equation is
\[
  (I-A)x=b.
\]
It is soluble precisely for $|V|/k(A)$ translations $b$, and for each soluble $b$ the solution set is a coset of $\ker(I-A)$ of size $k(A)$.  Averaging over $A\in H$ gives \eqref{eq:affine-enumerator}; the derangement and moment formulae follow immediately.
\end{proof}

Let $N=p^e$, and let $H_h\le\F_N^\times$ be the subgroup of order $h$, where $h\mid N-1$.  Put
\[
  G_{N,h}=(\F_N,+)\rtimes H_h.
\]

\begin{corollary}
\label{cor:scalar-profile}
In the natural action of $G_{N,h}$ on $\F_N$, the only nonzero fixed-point probabilities are
\begin{equation}
  \rho_0=\frac{N-1}{Nh},
  \qquad
  \rho_1=\frac{h-1}{h},
  \qquad
  \rho_N=\frac1{Nh}.
  \label{eq:scalar-profile}
\end{equation}
The rank of the action is
\[
  1+\frac{N-1}{h}.
\]
In particular, for $h=N-1$ the action is sharply $2$-transitive and $\rho_0=1/N$.
\end{corollary}

\begin{proof}
If the scalar part is $1$, the identity translation fixes all $N$ points and the other $N-1$ translations fix none.  If the scalar part is not $1$, then $I-A$ is invertible, and every translation has exactly one fixed point.  The rank formula follows from the orbits of $H_h$ on $\F_N$: one orbit is $\{0\}$ and the remaining $(N-1)/h$ orbits lie in $\F_N^\times$.
\end{proof}

\begin{corollary}
\label{cor:semilinear-profile}
Let $m\ge1$, put $s=p^{\gcd(e,m)}$, $C=(N-1)/(s-1)$, and $\kappa=\gcd(h,C)$.  Consider the coset of affine semilinear transformations of $\F_N$
\[
  \mathcal C_m=\bigl\{x\mapsto \alpha x^{p^m}+b:
       \alpha\in H_h,\ b\in\F_N\bigr\}.
\]
In this coset the proportions of elements with $0$, $1$, and $s$ fixed points are respectively
\begin{equation}
  \frac{\kappa(s-1)}{hs},\qquad
  \frac{h-\kappa}{h},\qquad
  \frac{\kappa}{hs},
  \label{eq:semilinear-profile}
\end{equation}
and no other fixed-point counts occur.
\end{corollary}

\begin{proof}
For fixed $\alpha\in H_h$, the linearized fixed-point equation is
\[
  x^{p^m}=\alpha^{-1}x.
\]
It has a nonzero solution if and only if
$\Nm_{\F_N/\F_s}(\alpha^{-1})=1,$
and then its solution space has size $s$; otherwise its kernel is trivial.  Since the norm is exponentiation by $C=(N-1)/(s-1)$ on $\F_N^\times$, exactly $\kappa=\gcd(h,C)$ elements of $H_h$ satisfy the norm-one condition.  For those $\alpha$, a fraction $1/s$ of translations $b$ give an affine equation with $s$ fixed points and the remaining fraction give none.  For the other $h-\kappa$ values of $\alpha$, every translation gives exactly one fixed point.
\end{proof}

\section{The scalar-affine polynomial family}
\label{sec:lambda-monodromy}

Let $p$ be prime, let $N=p^e$, and let $h\mid N-1$.  Throughout the rest of the paper set
\begin{equation}
  A=\frac{N-1}{h},
  \qquad
  \Lambda_{N,h}(U)=U(U^A-1)^h\in\F_p[U].
  \label{eq:Lambda-def}
\end{equation}

\begin{remark}
\label{rem:cyclotomic}
Fix $Q=p^m$ and put $\beta_Q=\gcd(A,Q-1)$.  On $\F_Q^\times$, the map $u\mapsto u^A$ is constant on the cosets of the subgroup
\[
  K_A=\{u\in\F_Q^\times:u^A=1\},\qquad |K_A|=\beta_Q.
\]
If $u^A=c$ on such a coset, then
\[
  \Lambda_{N,h}(u)=(c-1)^h u.
\]
Thus, over each fixed $\F_Q$, $\Lambda_{N,h}$ is a first-order generalized cyclotomic mapping of index $(Q-1)/\beta_Q$, with zero coefficients allowed.    The point of the present paper is the uniform all-extension law and the affine Galois-closure interpretation.
\end{remark}

\begin{theorem}
\label{thm:polynomial}
Let $k$ be an algebraically closed field of characteristic $p$.  The polynomial cover
\[
  \lambda_{N,h}:\Pone_U\longrightarrow\Pone_T,
  \qquad
  T=\Lambda_{N,h}(U),
\]
has degree $N$.  Its geometric Galois closure is the rational function field $k(z)$ with
\begin{equation}
  U=z^h,
  \qquad
  T=(z^N-z)^h.
  \label{eq:galois-closure}
\end{equation}
The geometric monodromy group in its degree $N$ action is
\begin{equation}
  G_{N,h}=(\F_N,+)\rtimes H_h,
  \label{eq:GNh}
\end{equation}
where $H_h\le\F_N^\times$ is the subgroup of order $h$.

The only possible finite branch value is $0$.  Above $0$, the point $U=0$ is unramified and the roots of $U^A=1$ have ramification index $h$.  The point $\infty_U$ lies above $\infty_T$ and is totally ramified.  Hence the branch values are $0$ and $\infty$ if $h>1$, and only $\infty$ if $h=1$.
\end{theorem}

\begin{proof}
Let
\[
  \widetilde G_{N,h}=\{z\mapsto \alpha z+\beta:\alpha\in H_h,\ \beta\in\F_N\}\le\Aut_k(k(z)).
\]
For $\beta\in\F_N$ one has
\[
  (z+\beta)^N-(z+\beta)=z^N-z,
\]
and for $\alpha\in H_h$ one has
\[
  (\alpha z)^N-\alpha z=\alpha(z^N-z).
\]
Since $\alpha^h=1$, the function $T=(z^N-z)^h$ is invariant under $\widetilde G_{N,h}$.  By Artin's fixed-field theorem for finite groups of field automorphisms,
\[
  [k(z):k(z)^{\widetilde G_{N,h}}]=|\widetilde G_{N,h}|=Nh.
\]
On the other hand, the degree of $T$ as a rational function of $z$ is $Nh$, so
\[
  [k(z):k(T)]=Nh.
\]
Since $T$ is invariant, $k(T)\subseteq k(z)^{\widetilde G_{N,h}}$, and the two fields have the same degree in $k(z)$.  Hence
\[
  k(z)^{\widetilde G_{N,h}}=k(T).
\]
The subgroup $H_h$ fixes $U=z^h$, and $[k(z):k(U)]=h=|H_h|$, so $k(z)^{H_h}=k(U)$.  Finally,
\[
  T=(z^N-z)^h=z^h(z^{N-1}-1)^h=U(U^A-1)^h.
\]
The Galois cover $k(z)/k(T)$ contains $k(U)$.  The subgroup $H_h$ is the stabilizer of the sheet $U=z^h$, and the core of $H_h$ in $\widetilde G_{N,h}$ is trivial: in the affine action it is the intersection of all point stabilizers.  Hence the induced action on the cosets of $H_h$ is faithful and is the natural affine action of $(\F_N,+)\rtimes H_h$ on $\F_N$.  Therefore no proper subextension of $k(z)/k(T)$ can be the normal closure of $k(U)/k(T)$, and $k(z)$ is the Galois closure of $k(U)/k(T)$.

For the finite ramification, compute
\[
  \Lambda_{N,h}'(U)
  =(U^A-1)^{h-1}\bigl(U^A-1+hAU^A\bigr).
\]
Since $hA=N-1=-1$ in characteristic $p$, this becomes
\[
  \Lambda_{N,h}'(U)=-(U^A-1)^{h-1}.
\]
If $h=1$, this derivative is $-1$, so there is no finite ramification.  If $h>1$, the finite critical points are exactly the roots of $U^A=1$, and they all map to $0$.  If $a^A=1$, then $a\ne0$, and $U^A-1$ has a simple zero at $a$ because $p\nmid A$.  Hence locally
\[
  \Lambda_{N,h}(U)=U(U^A-1)^h
\]
has order $h$ at $a$, so the ramification index is $h$.  The point at infinity is the unique pole of the polynomial and is totally ramified.
\end{proof}

\begin{proposition}
\label{prop:inertia}
In the Galois cover
\[
  \Pone_z\longrightarrow \Pone_T,
  \qquad T=(z^N-z)^h,
\]
the points above $T=0$ are the points $z=\beta$ with $\beta\in\F_N$.  Their inertia groups are the conjugates
\[
  H_{h,\beta}=\{z\mapsto \alpha z+(1-\alpha)\beta:\alpha\in H_h\},
\]
all of order $h$.  For $h=1$ these inertia groups are trivial, so $T=0$ is not a branch value.  The point $z=\infty$ is the unique point above $T=\infty$, and its inertia group is the whole affine group $G_{N,h}$.  Thus the finite inertia is tame, while the inertia at infinity has wild translation subgroup $(\F_N,+)$.
\end{proposition}

\begin{proof}
The equation $T=0$ on the Galois closure is $z^N-z=0$, so its solutions are precisely $\F_N$.  Since $(z^N-z)'=-1$, the function $z^N-z$ has a simple zero at each $\beta\in\F_N$, and the additional $h$-th power gives ramification index $h$.  Since $k$ is algebraically closed, residue field extensions at geometric points are trivial, and the inertia group is the stabilizer of the point in the Galois group.  The stabilizer in $G_{N,h}$ of $\beta$ is the displayed conjugate of $H_h$, so it is the inertia group at $z=\beta$.  The function $(z^N-z)^h$ has a unique pole, at $z=\infty$, of order $Nh=|G_{N,h}|$; hence $\infty$ is totally ramified in the Galois cover and its inertia group is all of $G_{N,h}$.
\end{proof}

\begin{proposition}
\label{prop:arithmetic-monodromy}
Let $q=p^m$.  Over $\F_q(T)$, the Galois closure of $T=\Lambda_{N,h}(U)$ is
\begin{equation}
  (\F_q\cdot\F_N)(z),
  \qquad
  T=(z^N-z)^h.
  \label{eq:arith-closure}
\end{equation}
Its constant field is
\[
  \F_q\cdot\F_N=\F_{p^{\operatorname{lcm}(m,e)}}.
\]
The arithmetic monodromy group is
\begin{equation}
  G_{N,h}\rtimes\Gal(\F_q\cdot\F_N/\F_q),
  \label{eq:arith-group}
\end{equation}
where $\Gal(\F_q\cdot\F_N/\F_q)$ acts on $\F_N$ by Frobenius and hence by semilinear automorphisms on $G_{N,h}$.  In particular, the geometric and arithmetic monodromy groups coincide over $\F_q$ if and only if $\F_N\subseteq\F_q$.
\end{proposition}

\begin{proof}
Over $\F_q\cdot\F_N$, Theorem~\ref{thm:polynomial} gives the Galois closure as $(\F_q\cdot\F_N)(z)$, and the geometric deck transformations are exactly
\[
  z\mapsto\alpha z+\beta,
  \qquad \alpha\in H_h,\quad \beta\in\F_N.
\]
Conversely, an arithmetic Galois closure over $\F_q(T)$ contains, after base change to an algebraic closure, all conjugates $z+\beta$ of a root of $Z^N-Z=w$; their differences are the constants $\beta\in\F_N$.  Hence its constant field contains $\F_q\cdot\F_N$, and no larger constant field is present in the rational field $(\F_q\cdot\F_N)(z)$.  This proves \eqref{eq:arith-closure}.

The quotient by the geometric group is the constant Galois group.  Arithmetic Frobenius conjugates
\[
  z\mapsto\alpha z+\beta
  \quad\text{to}\quad
  z\mapsto\alpha^q z+\beta^q,
\]
which preserves $H_h$ because $H_h$ is the unique subgroup of $\F_N^\times$ of order $h$.  This gives the semidirect product \eqref{eq:arith-group}.
\end{proof}

\section{Exact value distributions over all extensions}
\label{sec:exact}

Let $Q=p^m$.  Set
\begin{equation}
  s=p^{\gcd(e,m)},
  \qquad
  C=\frac{N-1}{s-1},
  \qquad
  \kappa=\gcd(h,C).
  \label{eq:parameters}
\end{equation}
Here $s=\#(\F_N\cap\F_Q)$, and $s-1$ divides $N-1$, so $C$ is an integer.  We first record a small arithmetic simplification which also makes the integrality of the formulae below transparent.

\begin{lemma}
\label{lem:parameters}
With the notation above,
\begin{equation}
  \gcd(A,Q-1)=\frac{\kappa(s-1)}h.
  \label{eq:beta-simplified}
\end{equation}
In particular, the integers
\[
  \frac{\kappa(s-1)}h,
  \qquad
  \frac{\kappa(Q-1)}h,
  \qquad
  \frac{\kappa(Q-s)}{hs}
\]
are well defined.  Moreover, for every $\gamma\in\F_N^\times$, the equation
\begin{equation}
  x^{p^m}=\gamma x
  \label{eq:semilinear-equation}
\end{equation}
on $\F_N$ has $s$ solutions if $\Nm_{\F_N/\F_s}(\gamma)=1$, and only the zero solution otherwise.
\end{lemma}

\begin{proof}
Write
\[
  N-1=(s-1)C,
  \qquad
  Q-1=(s-1)D.
\]
Since $\gcd(N-1,Q-1)=s-1$, we have $\gcd(C,D)=1$.  Put $h=\kappa h_1$ and $C=\kappa C_1$ with $\gcd(h_1,C_1)=1$.  The divisibility $h\mid N-1=(s-1)\kappa C_1$ gives $h_1\mid s-1$; write $s-1=h_1L$.  Then
\[
  A=\frac{N-1}{h}=LC_1,
  \qquad
  Q-1=h_1LD.
\]
Since $\gcd(C_1,h_1D)=1$, we obtain
\[
  \gcd(A,Q-1)=L=\frac{\kappa(s-1)}h,
\]
which proves \eqref{eq:beta-simplified}.  The integrality of $\kappa(s-1)/h=L$ and $\kappa(Q-1)/h=(Q-1)/h_1$ is immediate, because $h_1\mid s-1\mid Q-1$.  Finally, if $Q=s^d$, then
\[
  D=\frac{Q-1}{s-1}=1+s+\cdots+s^{d-1},
\]
so $D-1$ is divisible by $s$, and
\[
  \frac{\kappa(Q-s)}{hs}=\frac{L(D-1)}s
\]
is an integer.

For \eqref{eq:semilinear-equation}, nonzero solutions are the solutions of $x^{p^m-1}=\gamma$.  On the cyclic group $\F_N^\times$, the image of $x\mapsto x^{p^m-1}$ is the kernel of the norm $\Nm_{\F_N/\F_s}$.  Each element in this image has exactly $s-1$ nonzero preimages.  Adding $x=0$ gives $s$ solutions in the norm-one case and only zero otherwise.
\end{proof}

\begin{theorem}
\label{thm:exact-general}
For $t\in\F_Q$, let
\[
  r_Q(t)=\#\{u\in\F_Q:\Lambda_{N,h}(u)=t\}.
\]
Then
\begin{equation}
  r_Q(0)=1+\frac{\kappa(s-1)}h.
  \label{eq:zero-fibre}
\end{equation}
Among nonzero target values, the only possible fibre sizes are $0$, $1$, and $s$, and their exact numbers are
\begin{align}
  M_0(Q)&=\#\{t\in\F_Q^\times:r_Q(t)=0\}
     =\frac{\kappa(s-1)Q}{hs},
  \label{eq:general-A0}\\
  M_1(Q)&=\#\{t\in\F_Q^\times:r_Q(t)=1\}
     =\frac{(h-\kappa)(Q-1)}h,
  \label{eq:general-A1}\\
  M_s(Q)&=\#\{t\in\F_Q^\times:r_Q(t)=s\}
     =\frac{\kappa(Q-s)}{hs}.
  \label{eq:general-As}
\end{align}
\end{theorem}

\begin{proof}
The zero fibre consists of $u=0$ and the $\F_Q$-roots of $u^A=1$.  Equation \eqref{eq:zero-fibre} follows from Lemma~\ref{lem:parameters}.

\smallskip
\noindent\emph{Step 1: lifting the nonzero fibres.}
Let $t\ne0$ and choose $w$ in an algebraic closure with $w^h=t$.  Put
\[
  S_w=\{z:z^N-z=w\}.
\]
After choosing the representative $w$ in its $H_h$-orbit, the map $z\mapsto z^h$ is a bijection from $S_w$ onto the set of geometric roots $u$ of $\Lambda_{N,h}(u)=t$.  Indeed, if $z^N-z=w$, then
\[
  \Lambda_{N,h}(z^h)=z^h(z^{N-1}-1)^h=(z^N-z)^h=t.
\]
Conversely, if $u$ is a root and $z^h=u$, then $(z^N-z)^h=t=w^h$; multiplying $z$ by a unique element of $H_h$ gives $z^N-z=w$.  Two points of $S_w$ with the same $h$-th power differ by a scalar in $H_h$, and the equality $z^N-z=w$ forces that scalar to be $1$.

\smallskip
\noindent\emph{Step 2: the Kummer class and Frobenius.}
Let
\[
  \gamma=w^{Q-1}\in H_h.
\]
Replacing $w$ by $\xi w$ with $\xi\in H_h$ replaces $\gamma$ by $\xi^{Q-1}\gamma$.  Thus the class of $\gamma$ in
\[
  H_h/\{\xi^{Q-1}:\xi\in H_h\}
\]
is attached to $t$.  This class is the Kummer class of $t$: more precisely, it is the value on the arithmetic Frobenius of the cocycle representing the image of $t$ under the connecting homomorphism associated with the Kummer sequence
\[
  1\to\mu_h\to\mathbb G_m\xrightarrow{x\mapsto x^h}\mathbb G_m\to1.
\]
Since $\F_Q^\times$ is cyclic, the number of such Kummer classes is
$g_h=\gcd(h,Q-1)$, and each class contains $(Q-1)/g_h$ target values.

The $Q$-power Frobenius maps $S_w$ to $S_{\gamma w}$, so
\[
  \varphi_t:S_w\longrightarrow S_w,
  \qquad
  z\longmapsto\gamma^{-1}z^Q
\]
is well defined.  After an affine identification $S_w\simeq\F_N$, its linear part is
\begin{equation}
  x\longmapsto \gamma^{-1}x^{p^m}.
  \label{eq:linear-part}
\end{equation}
A point $u=z^h$ lies in $\F_Q$ if and only if $z^Q=\eta z$ for some $\eta\in H_h$.  Comparing $z^Q\in S_{w^Q}=S_{\gamma w}$ with $z^Q\in S_{\eta w}$ gives $\eta=\gamma$.  Hence the $\F_Q$-rational roots above $t$ are exactly the fixed points of $\varphi_t$.

\smallskip
\noindent\emph{Step 3: the norm criterion.}
By Lemma~\ref{lem:parameters}, the linearized fixed-point equation has kernel of size $s$ precisely when
\begin{equation}
  \Nm_{\F_N/\F_s}(\gamma)=1,
  \label{eq:norm-condition}
\end{equation}
and has trivial kernel otherwise.  Therefore a nonzero target has exactly one preimage unless \eqref{eq:norm-condition} holds; under \eqref{eq:norm-condition} it has either no preimages or $s$ preimages.

\smallskip
\noindent\emph{Step 4: counting the Kummer classes.}
The norm is exponentiation by
\[
  C=\frac{N-1}{s-1}.
\]
Thus
\[
  K_0=\{\gamma\in H_h:\Nm_{\F_N/\F_s}(\gamma)=1\}
\]
has order $\kappa=\gcd(h,C)$.  The subgroup $\{\xi^{Q-1}:\xi\in H_h\}$ is contained in $K_0$, because
\[
  \gcd(N-1,Q-1)=s-1
  \quad\text{and hence}\quad
  N-1\mid C(Q-1).
\]
It has order $h/g_h$, so the number of Kummer classes satisfying the norm condition is $\kappa/(h/g_h)$.  Since each class contains $(Q-1)/g_h$ target values, the number $B$ of nonzero target values satisfying the norm condition is
\begin{equation}
  B=\frac{\kappa(Q-1)}h.
  \label{eq:B-count}
\end{equation}
All other nonzero target values have exactly one preimage, which gives \eqref{eq:general-A1}.  Among the $B$ remaining values, let $M_s$ have $s$ preimages and $M_0$ have none.  Counting nonzero domain points and using Lemma~\ref{lem:parameters} gives
\[
  M_1+sM_s=Q-1-\frac{\kappa(s-1)}h.
\]
Together with $M_0+M_s=B$, this gives \eqref{eq:general-A0} and \eqref{eq:general-As}.
\end{proof}

\begin{remark}
The proof is deliberately finite-field exact.  The asymptotic Chebotarev theorem predicts the leading fixed-point proportions when $\F_N\subseteq\F_Q$; Theorem~\ref{thm:exact-general} refines this to an exact formula for every extension field, including the nonregular constant field cases.
\end{remark}

\section{Fibre enumerators, moments, and extremal value sets}
\label{sec:fibre}

Define the fibre enumerator
\[
  \mathcal E_{N,h,Q}(Z)=\sum_{t\in\F_Q}Z^{r_Q(t)}.
\]

\begin{corollary}
\label{cor:fibre-enumerator}
With the notation of Theorem~\ref{thm:exact-general},
\begin{equation}
  \mathcal E_{N,h,Q}(Z)
  =Z^{1+\kappa(s-1)/h}+M_0(Q)+M_1(Q)Z+M_s(Q)Z^s.
  \label{eq:fibre-enumerator}
\end{equation}
Consequently,
\begin{equation}
  \#\Lambda_{N,h}(\F_Q)
  =Q-M_0(Q)
  =Q\left(1-\frac{\kappa(s-1)}{hs}\right).
  \label{eq:general-support}
\end{equation}
For every fixed integer $a\ge1$,
\begin{equation}
  \sum_{t\in\F_Q}r_Q(t)^a
  =\left(1+\frac{\kappa(s-1)}h\right)^a
   +\frac{(h-\kappa)(Q-1)}h
   +s^{a-1}\frac{\kappa(Q-s)}h.
  \label{eq:all-fibre-moments}
\end{equation}
\end{corollary}

\begin{proof}
This is a direct repackaging of Theorem~\ref{thm:exact-general}.
\end{proof}

\begin{corollary}
\label{cor:collision}
Let $U_1,U_2$ be independent uniformly distributed elements of $\F_Q$, and put $Y_i=\Lambda_{N,h}(U_i)$.  Then
\begin{equation}
  \Pr(Y_1=Y_2)
  =\frac1{Q^2}\left[
    \left(1+\frac{\kappa(s-1)}h\right)^2
    +\frac{(h-\kappa)(Q-1)}h
    +\frac{s\kappa(Q-s)}h
  \right].
  \label{eq:collision}
\end{equation}
Equivalently, the squared $L^2$-distance from the uniform distribution on $\F_Q$ is
\begin{equation}
  \sum_{t\in\F_Q}\left(\frac{r_Q(t)}Q-\frac1Q\right)^2
  =\Pr(Y_1=Y_2)-\frac1Q.
  \label{eq:l2-distance}
\end{equation}
\end{corollary}

\begin{proof}
The collision probability is $Q^{-2}\sum_t r_Q(t)^2$; apply \eqref{eq:all-fibre-moments} with $a=2$.  The second identity is the standard expansion of the squared $L^2$-distance.
\end{proof}

\begin{corollary}
\label{cor:regular-exact}
Assume $\F_N\subseteq\F_Q$.  Then $s=N$, $\kappa=1$, and $r_Q(0)=A+1$.  Among nonzero target values,
\begin{align*}
  \#\{t:r_Q(t)=0\}&=\frac{(N-1)Q}{Nh},\\
  \#\{t:r_Q(t)=1\}&=\frac{(h-1)(Q-1)}h,\\
  \#\{t:r_Q(t)=N\}&=\frac{Q-N}{Nh}.
\end{align*}
Consequently,
\begin{equation}
  \#\Lambda_{N,h}(\F_Q)=Q\left(1-\frac{N-1}{Nh}\right).
  \label{eq:regular-support}
\end{equation}
\end{corollary}

\begin{proof}
Substitute $s=N$ into Theorem~\ref{thm:exact-general}.  Here $C=1$, hence $\kappa=1$, and $\kappa(s-1)/h=A$.
\end{proof}

\begin{corollary}
\label{cor:full-exact}
Let $h=N-1$ and write
\[
  \Lambda_N(U)=U(U-1)^{N-1}.
\]
For $Q=p^m$, put $s=p^{\gcd(e,m)}$.  Then $r_Q(0)=2$.  Among nonzero target values,
\begin{align*}
  \#\{t:r_Q(t)=0\}&=\frac Qs,\\
  \#\{t:r_Q(t)=1\}&=\frac{(Q-1)(s-2)}{s-1},\\
  \#\{t:r_Q(t)=s\}&=\frac{Q-s}{s(s-1)}.
\end{align*}
Therefore
\begin{equation}
  \#\Lambda_N(\F_Q)=Q\left(1-\frac1s\right).
  \label{eq:full-support}
\end{equation}
\end{corollary}

\begin{proof}
This is Theorem~\ref{thm:exact-general} with $A=1$, $h=N-1$, and $\kappa=C=(N-1)/(s-1)$.
\end{proof}

We now compare the exact formula with the upper bound \eqref{eq:intro-wsc}.

\begin{theorem}
\label{thm:wsc-gap}
Assume $Q\ge N$.  The polynomial $\Lambda_{N,h}$ is not a permutation polynomial of $\F_Q$.  Its exact defect from the Wan--Shiue--Chen upper bound for non-permutation polynomials of degree $N$ is
\begin{equation}
  \left(Q-\left\lceil\frac{Q-1}{N}\right\rceil\right)-\#\Lambda_{N,h}(\F_Q)
  =
  \frac{\kappa(s-1)Q}{hs}
  -\left\lceil\frac{Q-1}{N}\right\rceil.
  \label{eq:wsc-defect-general}
\end{equation}
If $\F_N\subseteq\F_Q$, then this simplifies to
\begin{equation}
  \left(Q-\left\lceil\frac{Q-1}{N}\right\rceil\right)-\#\Lambda_{N,h}(\F_Q)
  =\frac{Q(N-1-h)}{Nh}.
  \label{eq:wsc-defect-constant-field}
\end{equation}
Thus, over fields containing $\F_N$, equality with the Wan--Shiue--Chen bound occurs within this scalar-affine family if and only if $h=N-1$.
\end{theorem}

\begin{proof}
The polynomial is not a permutation polynomial because Lemma~\ref{lem:parameters} shows that $\kappa(s-1)/h$ is a positive integer, so $r_Q(0)=1+\kappa(s-1)/h\ge2$.  Equation \eqref{eq:wsc-defect-general} follows by subtracting \eqref{eq:general-support} from the bound \eqref{eq:intro-wsc}.  If $\F_N\subseteq\F_Q$, then $s=N$, $\kappa=1$, and, since $Q$ and $N$ are powers of the same prime with $N\mid Q$, $\left\lceil(Q-1)/N\right\rceil=Q/N$.  This gives \eqref{eq:wsc-defect-constant-field}.
\end{proof}

\begin{corollary}
\label{cor:wsc-extremal}
Assume $\F_N\subseteq\F_Q$.  Then
\[
  \Lambda_N(U)=U(U-1)^{N-1}
\]
is not a permutation polynomial of $\F_Q$ and
\begin{equation}
  \#\Lambda_N(\F_Q)=Q-\frac QN
  =Q-\left\lceil\frac{Q-1}{N}\right\rceil.
  \label{eq:full-affine-extremal}
\end{equation}
Thus $\Lambda_N$ attains the Wan--Shiue--Chen upper bound over every finite field containing $\F_N$.
\end{corollary}

\begin{proof}
This is Theorem~\ref{thm:wsc-gap} in the case $h=N-1$.
\end{proof}

\begin{corollary}
\label{cor:full-affine-defect}
Let $h=N-1$ and let $Q=p^m\ge N$.  With $s=p^{\gcd(e,m)}$,
\[
  \left(Q-\left\lceil\frac{Q-1}{N}\right\rceil\right)-\#\Lambda_N(\F_Q)
  =
  \frac Qs-\left\lceil\frac{Q-1}{N}\right\rceil.
\]
This defect is zero if and only if $\F_N\subseteq\F_Q$.
\end{corollary}

\begin{proof}
Use Corollary~\ref{cor:full-exact}.  If $s=N$, equivalently $e\mid m$, the defect is zero.  Conversely, if $s<N$, then $Q/s>Q/N=\lceil(Q-1)/N\rceil$ because $Q\ge N$ and $N\mid Q$.
\end{proof}

\begin{table}[t]
\centering
\caption{Full-affine examples for $\Lambda_N(U)=U(U-1)^{N-1}$.  Here $Q=p^m$ and $s=p^{\gcd(e,m)}$.}
\label{tab:examples}
\begin{tabular}{cccccc}
\toprule
$p$ & $N$ & $Q$ & $s$ & missing values & value set size \\
\midrule
$2$ & $4$ & $16$ & $4$ & $4$ & $12$ \\
$2$ & $8$ & $64$ & $8$ & $8$ & $56$ \\
$3$ & $9$ & $729$ & $9$ & $81$ & $648$ \\
$5$ & $25$ & $625$ & $25$ & $25$ & $600$ \\
$3$ & $27$ & $729$ & $27$ & $27$ & $702$ \\
\bottomrule
\end{tabular}
\end{table}

\section{Coprime pullbacks to elliptic curves}
\label{sec:pullback}

The exact formula above is a polynomial statement over the affine line.  The same geometric monodromy can be transported to other bases by  ``coprime'' pullback.  We record this standard base-change consequence for the explicit scalar-affine cover constructed above.

\begin{proposition}
\label{thm:pullback}
Let $k$ be algebraically closed of characteristic $p$, let $E/k$ be an elliptic curve, and let
\[
  \tau:E\longrightarrow\Pone_T
\]
be a finite separable morphism of degree $d$.  Assume
\[
  \gcd(d,Nh)=1.
\]
Let $X_{N,h,\tau}$ be the normalization of the fibre product
\[
  E\times_{\Pone_T}\Pone_U,
\]
where $\Pone_U\to\Pone_T$ is the cover $T=\Lambda_{N,h}(U)$.  On the affine locus over $\Aone_T$, this fibre product is given by
\[
  \Lambda_{N,h}(U)=\tau(P),
  \qquad P\in E.
\]
Then the projection
\[
  \pi_{N,h,\tau}:X_{N,h,\tau}\longrightarrow E
\]
is a geometrically connected degree $N$ cover with geometric monodromy $G_{N,h}$ in its natural degree $N$ action.
\end{proposition}

\begin{proof}
Let $F=k(T)$ and view $K=k(E)$ as an $F$-field via the pullback
\[
  \tau^*:F=k(T)\hookrightarrow k(E),
  \qquad
  T\longmapsto \tau^*(T).
\]
We write $\tau$ also for the rational function $\tau^*(T)\in k(E)$.  Let $M=k(z)$ be the Galois closure from Theorem~\ref{thm:polynomial}.  Then $[M:F]=Nh$ and $[K:F]=d$.  Since $M/F$ is Galois, the degree of $M\cap K$ over $F$ divides both $Nh$ and $d$.  By the coprimality assumption,
\[
  M\cap K=F.
\]
Thus $M$ and $K$ are linearly disjoint over $F$.

Every intermediate field of $M/F$, in particular $k(U)$, is therefore linearly disjoint from $K$.  Hence
\[
  K\otimes_F k(U)
\]
is a field, namely $Kk(U)$.  This is the function field of the normalization of the fibre product, so the pullback is connected.  Since $k$ is algebraically closed, this is geometric connectedness.  Moreover,
\[
  [Kk(U):K]=[k(U):F]=N,
\]
which gives the degree.

The Galois closure of $Kk(U)/K$ is $MK$.  By linear disjointness, restriction gives an isomorphism
\[
  \Gal(MK/K)\simeq\Gal(M/F).
\]
Under this isomorphism the stabilizer of the sheet $Kk(U)$ corresponds to the stabilizer of $k(U)$.  Therefore the induced permutation action on the $N$ sheets is the same as for the original cover $\lambda_{N,h}$, namely the natural degree-$N$ action of $G_{N,h}$.
\end{proof}

\begin{proposition}
\label{prop:prescribed-degree}
Let $E/k$ be an elliptic curve with origin $\cO$.  For every integer $d\ge2$ there exists $\tau\in k(E)$ with pole divisor
\[
  (\tau)_\infty=d[\cO].
\]
It defines a morphism $E\to\Pone$ of degree $d$.  If $\operatorname{char}k\nmid d$, the morphism is separable.
\end{proposition}

\begin{proof}
By Riemann--Roch on a genus-one curve,
\[
  \ell(d[\cO])=d
  \qquad(d\ge1);
\]
see \cite[Theorem~1.5.15 and Corollary~1.6.8]{Stichtenoth2009}.  Hence
\[
  \ell(d[\cO])>\ell((d-1)[\cO])
\]
for $d\ge2$, so there exists a function with a pole of exact order $d$ at $\cO$ and no other poles.  Its polar divisor has degree $d$, hence the induced morphism $E\to\Pone$ has degree $d$.

If the morphism were inseparable in positive characteristic $p$, then $\tau\in k(E)^p$ because $k$ is perfect.  In particular all valuations of $\tau$ would be divisible by $p$.  Since $v_{\cO}(\tau)=-d$ and $p\nmid d$, this cannot happen.  In characteristic zero separability is automatic.  Thus the morphism is separable whenever $\operatorname{char}k\nmid d$.
\end{proof}

\begin{corollary}
\label{cor:existence-pullbacks}
Let $k$ be algebraically closed of characteristic $p$, let $E/k$ be an elliptic curve, and let $d\ge2$ satisfy $\gcd(d,Nh)=1$.  Then there exists a finite separable morphism
\[
  \tau:E\to\Pone
\]
of degree $d$.  For every such $\tau$, the normalization of the fibre product
\[
  \Lambda_{N,h}(U)=\tau(P)
\]
is geometrically connected over $E$, has degree $N$, and has geometric monodromy $G_{N,h}$ in its natural degree-$N$ action.
\end{corollary}

\begin{proof}
By Proposition~\ref{prop:prescribed-degree} there is a function $\tau\in k(E)$ with pole divisor $d[\cO]$, hence a morphism $E\to\Pone$ of degree $d$.  Since $\gcd(d,Nh)=1$ and $N=p^e$, we have $p\nmid d$, so the morphism is separable.  The final assertions follow from Proposition~\ref{thm:pullback}.
\end{proof}

\begin{proposition}
\label{prop:genus}
Assume the hypotheses of Proposition~\ref{thm:pullback}.  Suppose in addition that $\tau$ is unramified above all branch values of $\Lambda_{N,h}$, equivalently above $\infty$ and, when $h>1$, above $0$.  Then
\[
  g(X_{N,h,\tau})=1+d(N-1).
\]
\end{proposition}

\begin{proof}
By Riemann--Hurwitz formula \cite[Theorem~3.4.13]{Stichtenoth2009}, the original degree-$N$ cover $\Pone_U\to\Pone_T$ has total different degree $2N-2$:
\[
  -2=N(-2)+\deg\operatorname{Diff}(\lambda_{N,h}).
\]
Thus $\deg\operatorname{Diff}(\lambda_{N,h})=2N-2$.  The assumption that $\tau$ is unramified over the branch values means that, locally at every branch value, the base change of the corresponding dvr is unramified.  Unramified base change preserves different exponents; see \cite[Chapter~3, Section~3.4]{Stichtenoth2009} or \cite[Chapter~III, Sections~1--4]{SerreLocalFields}.  Hence each local contribution of the original cover appears $d$ times in the pullback, and the total different degree of $X_{N,h,\tau}\to E$ is $d(2N-2)$.

Since $E$ has genus one, Riemann--Hurwitz again gives
\[
  2g(X_{N,h,\tau})-2=d(2N-2),
\]
which is the stated formula.
\end{proof}

\begin{remark}
The functions supplied by Proposition~\ref{prop:prescribed-degree} need not satisfy the unramifiedness hypothesis of Proposition~\ref{prop:genus}; indeed a function with pole divisor $d[\cO]$ is ramified above $\infty$.  Proposition~\ref{prop:genus} is a conditional genus formula for choices of $\tau$ whose branch locus avoids the branch values of $\lambda_{N,h}$.
\end{remark}

\begin{corollary}
\label{cor:elliptic-cheb}
Assume that the cover in Proposition~\ref{thm:pullback} is defined over $\F_q$, and let $k_L$ be the constant field of its arithmetic Galois closure.  If $k_L\subseteq\F_Q$, then, for every $j$, the number $M_j(Q)$ of unramified points $P\in E(\F_Q)$ for which the fibre of $\pi_{N,h,\tau}$ above $P$ has exactly $j$ rational points satisfies
\[
  M_j(Q)=\rho_j(G_{N,h},\F_N)\#E(\F_Q)+O_{N,h,\tau}(Q^{1/2}),
\]
where $\rho_j(G_{N,h},\F_N)$ denotes the fixed-point probability in the natural action on $\F_N$.
In particular,
\[
  \#\pi_{N,h,\tau}(X_{N,h,\tau}(\F_Q))
  =\left(1-\frac{N-1}{Nh}\right)\#E(\F_Q)+O_{N,h,\tau}(Q^{1/2}).
\]
For $h=N-1$ the main term is $(1-1/N)\#E(\F_Q)$.  If the arithmetic constant field of the pullback is $\F_q\cdot\F_N$, it is enough to assume $\F_N\subseteq\F_Q$.
\end{corollary}

\begin{proof}
Apply Theorem~\ref{thm:coset-chebotarev} to the pullback cover.  The condition $k_L\subseteq\F_Q$ makes the sampled Frobenius coset the geometric one.  The fixed-point probabilities are those of Corollary~\ref{cor:scalar-profile}.
\end{proof}

\section{Comparison with symmetric monodromy}
\label{sec:symmetric}

Let $D_n$ be the number of derangements in the symmetric group $S_n$:
\[
  D_n=n!\sum_{i=0}^n\frac{(-1)^i}{i!}.
\]
For the natural action of $S_n$, the number of permutations with exactly $j$ fixed points is $\binom njD_{n-j}$.  Therefore a degree-$n$ cover with regular Galois closure and natural monodromy $S_n$ satisfies
\[
  \rho_j=\frac{\binom njD_{n-j}}{n!},
  \qquad 0\le j\le n.
\]
In particular,
\[
  \frac{\#\pi(X(\F_Q))}{\#Y(\F_Q)}
  =1-\frac{D_n}{n!}+O_\pi(Q^{-1/2}).
\]
Since $D_n/n!\to e^{-1}$, the support density for natural symmetric monodromy tends to $1-e^{-1}$; for a much broader fixed-point theory in symmetric groups see  \cite[Theorem~1.1]{DiaconisFulmanGuralnick2008}.

By contrast, the full affine group $\AGL(1,N)$ has derangement proportion exactly $1/N$, and the polynomials $\Lambda_N$ realize this optimal fixed-point profile exactly over every field containing $\F_N$.  The comparison is summarized as follows.  The third fixed-point moment in the table is
\[
  |G|^{-1}\sum_{g\in G}\#\Fix(g)^3.
\]

\begin{center}
\small
\begin{tabular}{@{}llll@{}}
\toprule
monodromy & derangements & support & $\mathbb E(F^3)$ \\
\midrule
$\AGL(1,n)$, $n=p^a$ & $1/n$ & $1-1/n$ & $n+2$ \\
$S_n$ & $\sim e^{-1}$ & $\sim1-e^{-1}$ & $5$ for $n\ge3$ \\
\bottomrule
\end{tabular}
\end{center}

Thus increasing degree while retaining generic symmetric monodromy does not make a one-dimensional cover nearly surjective.  The scalar-affine family gives an explicit finite-field model for the opposite mechanism: special affine monodromy produces optimal support, and our paper computes the resulting fibre distribution exactly.

\end{document}